\documentclass[12pt]{amsart}
\usepackage{amssymb,graphicx,amsmath}

\newtheorem{thm}{Theorem}
\newtheorem{prop}{Proposition}

\newtheorem{lem}{Lemma}
\newtheorem{cor}{Corollary}

\newtheorem*{mainthm}{Main Theorem}

\theoremstyle{definition}

\newcommand{\f}[1]{\mathbb{#1}}

\def\bN{\mathbb{N}}

\def\om{\omega}
\def\Om{\Omega}
\newcommand{\m}[1]{\mathcal{#1}}

\newcommand{\gpr}[2]{ \left\langle #1 \mid #2 \right\rangle}
\newcommand{\gp}[1]{ \left\langle #1 \right\rangle}

\usepackage[ marginparsep=5mm,heightrounded,centering,margin=3.47cm]{geometry}

\title[]{On the condensation property of the Lamplighter groups and groups of Intermediate growth}

\author{Mustafa G\"okhan  Benli, Rostislav Grigorchuk}

\date{05/07/2014\\ The  authors were supported by NSF grant DMS-1207699}

\begin{document}

\begin{abstract}
The aim of this short note is to   revisit  
some old  results about  groups of  intermediate  growth  and  groups  of  the 
lamplighter  type and to show that the Lamplighter group $L=\mathbb{Z}_2\wr \mathbb{Z}$  
is a condensation group and  has  a minimal  presentation  
by  generators  and  relators. 
The  condensation  property  is  achieved  by  showing  that $L$  belongs  
to  a  Cantor  subset  of  the space $\mathcal{M}_2$ 
of  marked 2-generated groups consisting  mostly of  groups of  intermediate  growth.
\end{abstract}

 \maketitle
 
 \section{Introduction}

The modern development of group theory requires significant use of methods of geometry, topology, probability 
and measure theory, the theory of models etc. The space $\m{M}_k$ of marked
$k$-generated groups, introduced in \cite{MR764305} plays an important role in this development. It is a compact totally
disconnected metrizable space and it is important to know which groups belong to its perfect kernel (or condensation part), 
which is homeomorphic to a Cantor set. Groups in the perfect kernel are called \textit{condensation groups}.
The aim of this note is to revisit some results of \cite{MR764305} and to use them to show that the so called
Lamplighter group $L=\f{Z}_2 \wr \f{Z}$, which is a popular object of study (see for example 
\cite{MR1866850,grigorch_kravchenko}) 
belongs to a Cantor set and hence is a condensation group.

 Let $\Omega=\{0,1,2\}^\f{N}$ be the set of all infinite sequence over $\{0,1,2\}$ with the product (Tychonoff) topology.

\begin{mainthm}  There exists a subset 
$\m{L}=\{(L_\om,T_\om)\mid \om \in \Om\}\subset \m{M}_2$ with the following properties:
 
 \begin{itemize}  \item[a)] $\m{L}$ is homeomorphic to $\Omega$ (and hence is a Cantor set),
  \item[b)] If $\om \in \Om$ is not eventually constant, then $L_\om$ has intermediate growth.
  \item[c)] If $\om \in \Om$ is a constant sequence then $L_\om\cong L$.
  \item[d)] All groups in $\m{L}$ are condensation groups.
 \end{itemize}
\end{mainthm}

  A simple argument shows that a group possessing an infinite minimal presentation is a condensation group.
  Surprisingly, it was observed in \cite{JMJ:9135731} that there are finitely generated groups
  which do not have a minimal presentation.
  It follows from \cite{MR0120269} that groups of the form $H\wr G$ where 
  $H$ and $G$ are infinite and finitely generated are not finitely presented and it was observed in \cite{MR2895067}
  that such groups are condensation groups. It is probably well known (as indicated in 
  \cite{JMJ:9135731}) that the standard presentation 
    $$L=\gpr{s,t}{s^2,[s,s^{t^i}]\;i\ge 1}$$
    is minimal. A proof of this fact using ideas of \cite{MR0120269} is presented for completeness. This
    provides an alternative proof of the fact that $L$ is a condensation group.

An effective way to build large families of condensation groups is to construct closed
subsets $X\subset \m{M}_k,k\ge2$ homeomorphic to a Cantor set. 
Such families were constructed in \cite{MR764305,MR784354,MR1760424,MR2318543}. It will be interesting to produce 
such families based on new ideas.

 \section{Preliminaries}

 For a topological space $X$, let $X'$ denote its set of accumulation points. For any ordinal
$\alpha$ define the spaces $X^{(\alpha)}$ inductively as follows:
$X^{(0)}=X,X^{(\alpha+1)}=\left( X^{(\alpha)} \right)'$ and $\displaystyle X^{(\lambda)}=\bigcap_{\beta<\lambda}X^{(\beta)}
$ if $\lambda$ is a limit ordinal. If $X$ is a Polish space, (i.e., a completely 
metrizable, separable space) for   some countable  ordinal $\alpha_0$ we 
will have $X^{(\alpha_0)}=X^{(\alpha)}$ for all $\alpha\ge \alpha_0$ (see \cite[Theorem 6.1]{MR1321597}).
The least ordinal
with this property is called the \emph{Cantor-Bendixon rank} of $X$ and will be denoted
by $rk_{CB}(X)$.
The set $X^{(\alpha_0)}$
is called the \emph{perfect kernel} (or condensation part) of $X$ which will be denoted by
$\kappa(X)$. Note that if nonempty, $\kappa(X)$ is homeomorphic to a Cantor set and $\kappa(X)$ is empty
if and only if $X$ is countable.
Points in $\kappa(X)$ are called condensation points and can be characterized as 
points for which every open neighborhood is uncountable (see \cite[I.6]{MR1321597}).

 Let $\m{M}_k$ denote the space marked groups consisting of pairs $(G,S)$ where
 $G$ is a group and $S$ is an ordered set of (not necessarily distinct) set of $k$ generators.
 Two marked groups $(G,S)$ and $(H,T)$ in $\m{M}_k$ are identified whenever the 
 map $s_i \mapsto t_i, i=1,\ldots,k$ extends to an isomorphism. Two points $(G,S)$ and $(H,T)$  
 are of distance  $\le 2^{-N}$
if the Cayley graphs of $(G,S)$ and $(H,T)$ have isomorphic balls of radius $N$. 
This (ultra) metric makes $\m{M}_k$ into a compact, totally disconnected, separable space.
It follows from the definition that a sequence $(G_n,S_n)\in \m{M}_k$ converges to $(G,S)\in \m{M}_k$, if and only
if, for every element $w\in F_k$ (the free group of rank $k$), there exists $N=N_w\ge0$, such that the
the relation $w=1$ holds in $G$
if and only if it holds in $G_n$ for $n 
\ge N$.

An important  problem of geometric group theory (raised in \cite{MR2195454}) is the 
identification of $rk_{CB}(\m{M}_k)$
for $k\ge2$. It follows from \cite{MR2895067} that the lower bound
$rk_{CB}(\m{M}_k)> \om^\om,k\ge 2$ holds. By a classical result of B.H. Neumann \cite{zbMATH03026064} 
there exists uncountably many non-isomorphic 2-generated groups. Therefore $\kappa(\m{M}_k)$
is a Cantor set for all $k\ge 2$. A finitely generated group $G$ is called a \emph{condensation group}, if for some generating
set $S$ of size $k$ the pair $(G,S)$ belongs to $\kappa(\m{M}_k)$. It follows that this property does not 
depend on the generating set (see \cite[Lemma 1]{MR2278053}).

 In \cite{MR764305} the second author constructed 
 Cantor sets $  \m{  G}\subset \m{M}_k$  consisting essentially of groups of intermediate
growth.
Clearly, groups belonging to these families lie in the condensation part of $\m{M}_k$.
In general, it is a challenging problem to identify which groups are in the condensation
part. It is expected that every group of intermediate growth is a condensation group.
In contrast, it is easy to observe that virtually nilpotent 
groups are not condensation.
 In \cite{MR2895067,JMJ:9135731} condensation properties of metabelian groups were  considered and
 it was proven that restricted wreath products $H\wr G$ of two finitely generated infinite groups 
 are condensation groups \cite[Proposition 8.1]{MR2895067}. Also, by \cite{MR1760424}
 every non-elementary hyperbolic groups is a condensation group.
 
%

Let us briefly recall the groups constructed  in \cite{MR764305}. 
Although the original definition is in terms of measure preserving transformations
of the unit interval, we will give here a definition in terms of automorphisms of rooted trees.
Let   $\Om$ denote the set all infinite sequences over 
the alphabet $\{0,1,2\}$. 
We identify $\Om$ with
the product $\{0,1,2\}^{\bN}$ and endow it with the product topology. Let $\tau :\Om \to \Om$
be the shift transformation, i.e., $\tau(\om)_n=\om_{n+1}$.
For each 
$\om \in \Omega$ we will define a subgroup $G_\om$
of $Aut(\mathcal{T}_2)$, where the latter denotes the automorphism group of the binary
rooted tree $\m{T}_2$ whose vertices are identified with the set of finite sequences $\{0,1\}^\ast$.
Each group $G_\om$ is the subgroup generated by the four automorphisms denoted by
$a,b_\om,c_\om,d_\om$
whose actions onto the tree is as follows:

\smallskip

For  $v\in\{0,1\}^*$
$$a(0v)=1v \;\text{and} \; a(1v)=0v  $$

$$\begin{array}{llllll}
 b_\om(0v)=& 0 \beta(\om_1)(v) &   c_\om(0v)=& 0 \zeta(\om_1)(v)
  &  d_\om(0v)= &0 \delta(\om_1)(v) \\

   b_\om(1v)=& 1 b_{\tau (\om)}(v) &      c_\om(1v)= & 1 c_{\tau \om}(v)
   & d_\om(1v)= & 1 d_{\tau \om}(v), \\

\end{array}
$$
where
$$
\begin{array}{ccc}
 \beta(0)=a & \beta(1)=a & \beta(2)=e \\
  \zeta(0)=a & \zeta(1)=e & \zeta(2)=a \\
   \delta(0)=e & \delta(1)=a & \delta(2)=a \\
\end{array}
$$
and $e$ denotes the identity. 

Note that from the definition, the following relations are immediate:
$$
 a^2=b_\om^2=c_\om^2=d_\om^2=b_\om c_\om d_\om =e
$$
 
Denoting by $S_\om=\{a,b_\om,c_\om,d_\om\}$, we obtain a subset $\{(G_\om,S_\om) \mid \om \in \Om\}\subset \m{M}_4$.
In \cite{MR764305} it was observed that this subset is not closed. 
It was also shown in \cite{MR764305} that modifying countably many groups in this family, one obtains
a closed subset $\m{G}=\{(G_\om,S_\om)\mid \om \in \Om\}$ with the following properties:
 
 \begin{thm}[\cite{MR764305}]\label{grig84} \mbox{}
 \begin{itemize}
  \item[1)] $\m{G}$ is homeomorphic to $\Om$ via the map $\om \mapsto (G_\om,S_\om)$,
  \item[2)] If in $\om \in \Om$  all symbols $\{0,1,2\}$ appear infinitely often, then $G_\om$ is a 2-group,
  \item[3)] For $\om \in \Om$ which is not eventually constant (i.e., is not constant after some point), $G_\om$ has intermediate growth,
  \item[4)] If $\om \in \Om$ is eventually constant, then $G_\om$ is virtually metabelian of exponential growth.
 \end{itemize}

 \end{thm}

 \section{Proof of the Main Theorem}
 
 We start with the following basic lemma:
 
  \begin{lem} \label{lem1}
  Suppose that  $\{(G_n,S_n)\}$ is a sequence in $\m{M}_k$ converging to $(G,S)$.  Let
  $F_k$ be the free group of rank $k$, with basis $\{x_1,\ldots,x_k\}$ and
  let   $\pi:F_k\to G$ , $\pi_n :F_k \to G_n$ be the canonical maps.
   Given $w_1,\ldots,w_m \in F_k$, let  
   $T=\{\pi(w_1),\ldots,\pi(w_m)\}$ , $T_n=\{\pi_n(w_1),\ldots,\pi_n(w_m)\}$ and $ H=\gp{T}\le G,H_n=\gp{T_n}\le G_n$. Then
   the sequence $\{(H_n,T_n)\}$ converges to $(H,T)$ in $\m{M}_m$.
 \end{lem}
 
 \begin{proof}
Let $F_m$ be the free group of rank $m$ with basis $\{y_1,\ldots,y_m\}$ and let
   $\gamma:F_m\to H$  and $\gamma_n :F_k \to H_n$ be the canonical maps. Also, let
$p:F_m \to F_k$ be the group homomorphism defined by $p(y_i)=w_i\;,\;i=1,\ldots,m$. Note that we have
the following: $$\gamma_n = \pi_n \circ p \;\;\; \text{for every}\;\;n $$
and $$\gamma=\pi \circ p $$
It follows that, given $w\in F_m$, $w=1$ in $H$ if and only if $p(w)=1$ in $G$. This shows
that the sequence $\{(H_n,T_n)\}$ converges to $(H,T)$ in $\m{M}_m$.
 \end{proof}

 The following is a description of  the structure of the group $G_{000\ldots}$.

\begin{thm} \label{thm2}
 The group $ G_{000\ldots}$ is isomorphic to the group
 $L\rtimes \mathbb{Z}_2 $ where $L=\mathbb{Z}_2 \wr \mathbb{Z}$ is the Lamplighter group
 given by presentation $\gpr{s,t}{s^2,[s,s^{t^i}], i \ge 1
 }$ and   $\f{Z}_2$ acts on $L$ 
 by the automorphism 
$$ s \mapsto s^t$$
$$ t \mapsto t^{-1}$$
\end{thm}

\begin{proof}
 Let us denote  $ G_{000\ldots}$ by $G$ and denote
 its canonical generators by $a,b,c,d$. 
 Let $H$ be the subgroup of $G$ generated by the elements $b,c,d,b^a,c^a,d^a$.
 There exists an embedding  (see \cite{MR764305}) 
$$
\begin{array}{ccc}
 \psi : H & \rightarrow & G  \times  G \\
 b & \mapsto    & (a,b) \\
c & \mapsto    & (a,c) \\
d & \mapsto    & (1,d) \\\
b^a & \mapsto    & (b,a) \\
c^a & \mapsto    & (c,a) \\
d^a & \mapsto    & (d,1) \\
\end{array}
$$
Let $D=\gp{\gp{d}}$ be the normal closure of $d$ in $G$. 
By induction on word length one can see that $D$ is 
an abelian group (see \cite[Lemma 6.1]{MR764305}).

We claim that $D=\gp{d^g \mid g \in \gp{a,b}}$.
Let us denote the right hand side by $T$. Clearly $T$ is contained in $D$. 
It suffices to show $T$ is normal. 
Since $bcd=1$ it is enough to show that $(d^g)^c \in T$ for all $g\in \gp{a,b}$. 
By induction on $k$ one can see  that the following equality holds:
$$
\psi(d^{(ab)^n}) = \left\{
\begin{array}{cl}
   (1,d^{(ab)^k}) & , \;n=2k \\
   (d^{(ab)^ka},1) & , \;n=2k+1 \\
    \end{array}
\right.$$
We will show  by induction on $|g|$ that $(d^g)^c=(d^g)^b$. Suppose $|g|=1$, the case $g=b$ is obvious 
since $bcd=1$. If $g=a$, we have $\psi((d^a)^b)=\psi((d^a)^c)=(d^a,1)$ and hence $(d^g)^c=(d^g)^b$. Now
assume $|g|>1$. Since $d^b=d$ we can assume that $g$ starts with $a$. There are two cases, either $g=(ab)^n$
or $g=(ab)^na$ for some $n$. In the first case (using induction assumption)
$$
\psi((d^{(ab)^n})^c) = \left\{
\begin{array}{cl}
   (1,(d^{(ab)^k})^c)=(1,(d^{(ab)^k})^b) & , \;n=2k  \\
   ((d^{(ab)^ka})^a,1)=(d^{(ab)^k},1) & , \;n=2k+1\\
    \end{array}
\right.$$
and in the second case
$$
\psi((d^{(ab)^na})^c) = \left\{
\begin{array}{cl}
   ((d^{(ab)^k})^a,1)& , \;n=2k  \\
   (1,(d^{(ab)^ka})^c)=(1,(d^{(ab)^k})^b) & , \;n=2k+1\\
    \end{array}
\right.$$
In any case $\psi((d^g)^c)=\psi((d^g)^b)$.  This shows $T$ is normal and hence $D=T$.

Now
letting 
$$
t_n = \left\{
\begin{array}{cl}
   d^{(ab)^n} & n \geq 0  \\
   d^{(ab)^{-n-1}a} & n < 0\\
    \end{array}
\right.$$
we see that $T=\gp{t_n \mid n\in \f{Z}}$. Looking at $\psi(t_n)$ we see that the $t_n$ are mutually distinct, 
therefore $\displaystyle T\cong \prod_{\f{Z}} \f{Z}_2$.

Since $\psi((ab)^2)=(ba,ab)$ , it follows that the element $ab$ is of infinite order in $G$.  

We will show that the subgroups $D$ and $\gp{ab}$ intersect trivially. Suppose not, then $d^g=(ab)^n$ for some 
$g \in \gp{a,b}$ and $n\in \f{Z}$. Necessarily $n$ has to be even since left hand side of $d^g=(ab)^n$ has 
even number of $a$'s. 
If $n=2k$ then $\psi((ab)^{2k})=((ba)^k,(ab)^k)$ whereas $\psi(d^g)= (d^h,1)$ or $(1,d^h)$
for some element $h \in G$. It follows that $(ab)^k=1$ which is a contradiction since $ab$
has infinite order.

Now the subgroup $K = D \rtimes \gp{ab}$ is isomorphic to $(\f{Z}_2^{\infty})\rtimes \f{Z} \cong \f{Z}_2 \wr 
\f{Z}$ which is the Lamplighter group. This is true since we have 
$$t_n^{ab}=t_{n+1} , \quad n \in \f{Z}$$
and hence the generator $\gp{ab}$  acts on $D$ by shifting its generators.

Conjugating the generators of $K$ 
by the generators of $G$ we see that $K$ is a normal subgroup. The quotient  $G / D$ is isomorphic to the infinite 
dihedral group $D_{\infty}$ (see \cite[Lemma 6.1]{MR764305}) and maps onto the quotient $G/K$. The kernel of this homomorphism contains the image
of $ab$ in $G/D$. From this It follows that $K$ has index 2 in $G$. Hence we have $G = K \rtimes \gp{a} \cong L \rtimes \f{Z}_2$.
Identifying $s=d , t=ab$ we see that conjugation by $a$ gives the asserted automorphism of $K$.
\end{proof}

For $\om \in \Om$ let  $L_\om=\gp{d_\om,ab_\om}\le G_\om$. By virtue of the  relations 
$a^2=b_\om^2=c_\om^2=d_\om^2=b_\om c_\om d_\om=1$ we see that $L_\om$ is a normal subgroup of index 2
in $G_\om$ and hence share many properties with $G_\om$. Let us denote by $T_\om=\{d_\om,ab_\om\}$
and $\m{L}_\om=\{(L_\om,T_\om) \mid \om \in \Om\}\subset \m{M}_2$.

%

\medskip

\textbf{Proof of the main Theorem:}

\medskip

a)
 Consider the map $\phi : \Om \to \m{L}$ given by $\om \mapsto (L_\om,T_\om)$. $\phi$
 is continuous since, if $w_n$ converges to $w$, then by Theorem \ref{grig84} $(G_{\om_n},S_{\om_n})$
 converges to $(G_\om,S_\om)$ and hence by Lemma \ref{lem1} 
 $(L_{\om_n},T_{\om_n})$
 converges to $(L_\om,T_\om)$. To see that $\phi$ is injective: By \cite[Section 5]{MR764305},
 the following is true: Given $\om_1\neq \om_2$ in $\Om$, there exists $u\in F_4$ 
 (depending on $\om_1$ and $\om_2$), such that $u$ is trivial in $G_{\om_1}$ and nontrivial in 
 $G_{\om_2}$ (this amounts to saying that the map $a\mapsto a,b_{\om_1}\mapsto b_{\om_2},
 c_{\om_1}\mapsto c_{\om_2},d_{\om_1}\mapsto d_{\om_2}$ does not extend to an isomorphism
 from $G_{\om_1}$ to $G_{\om_2}$ i.e.,
 $(G_{\om_1},S_{\om_1})$ and 
 $(G_{\om_2},S_{\om_2})$ are distinct points in $\m{M}_4$).  One observes that
 such $u$ is a 2 power and (since $L_\om$ has index 2 in $G_\om$) its image in $G_\om$ 
 lies in  $L_\om$.
 Therefore the image of $u$ in $L_{\om_1}$ is trivial but its image in $L_{\om_2}$ is nontrivial
 which implies that $(L_{\om_1},T_{\om_1})$ and $(L_{\om_2},T_{\om_2})$ are distinct. This shows
 that $\phi$ is injective and by compactness we have that $\phi$ is a homeomorphism. 
 
b) This follows from  Theorem \ref{grig84} and the fact that $L_\om$ has finite index in $G_\om$.

c) By Theorem \ref{thm2} $L_{000\ldots}$ is isomorphic to $L$ and it is immediate 
from the definition of the groups that $G_{000\ldots}\cong G_{111\ldots}\cong G_{222\ldots}$
and $L_{000\ldots}\cong L_{111\ldots}\cong L_{222\ldots}$.

d)  This follows from part a).

As a corollary we obtain the following:

\begin{cor}\label{cor1}
 The Lamplighter group $L$ is a condensation group.
\end{cor}

\section{Minimal Presentations of the Lamplighter groups}

For a subset $A\subset G$ of a group let $\left<\left< A \right> \right>$ denote the normal subgroup
 generated by $A$.
  A presentation $\gp{X \mid R}$ is called \textit{minimal} if for every $r \in R$ we have 
  $r\notin \left<\left<R\setminus\{r\} \right> \right>$. The following is well known.

 \begin{prop}
  Let $\left< X \mid r_1,r_2,\ldots \right>$ be an infinite minimal presentation where $|X|=k$. Then the marked group 
  $(G,X)$ lies in the condensation part of $\m{M}_k$.
 \end{prop}
 
 \begin{proof}It is enough to show that any open ball around $(G,X)$  is uncountable.
  Let $B=B((G,X),2^{-N})$ be a ball of radius $2^{-N}$ around $(G,X)$. 
  A marked group $(H,T)\in\m{M}_k$ lies in $B$ if and only if  for all $w\in F_k$ such that $|w|\le 2N+1$, we have 
  $w=1\;\mbox{in}\;\;G \iff w=1\;\mbox{in}\;H.$
  Let $A=\{w\in F_k \mid |w|\le 2N+1\;\mbox{and}\;w=1\;\mbox{in}\;G\}$.
  Choose $M=M(N)\in \f{N}$  large enough so that $A\subset \left<\left< r_1,r_2,\ldots,r_M\right>\right>$.
  For any subset $U\subset\f{N}$ such that $\{1,2,\ldots,M\}\subset U$, let $(G_U,X)$ be the group
  $\left<X\mid r_i, i\in U \right>$. Clearly all $(G_U,X)\in B$ and since the initial presentation is minimal
  all of them are distinct marked groups. Hence $B$ is uncountable. 
 \end{proof}
We  will give an alternative proof of Corollary \ref{cor1} by showing that the  standard  presentation of $L$ is minimal.

  For a group $G$  and a subset 
  $S\subset G$ let $$T_S=\{(s_1g,s_2g)\mid s_1,s_2 \in S, g \in G\}\subset G\times G$$

\begin{thm}\cite{MR0120269}
 Let $G$ and $H$ be two groups and $S\subset G$ be a subset. Then there exists a group $W=W(H,G,S)$
 (called the circle product  of $G$ and $H$ with respect 
 to $S$) with the following properties:

 \begin{itemize}
 \item $W$ contains subgroups $H_g,g\in G$ all isomorphic to $H$,
  \item $W$ is generated by $G$ and $H_1$,
  \item The subgroup $K=\gp{H_g \mid g\in G}$ is normal in $W$and $W=K\rtimes G$,
  \item For $h_{g_1}\in H_{g_1}$ and $g_2\in G$  we have $h_{g_1}^{g_2}\in H_{g_1g_2}$,
  \item $[H_{g_1},H_{g_2}]=1$ if and only if $(g_1,g_2)\in T_S$.
 \end{itemize}
 \end{thm}
 Note that $W$ can also be realized  by using graph products: Let $\Gamma$ be the graph with vertex set
 $G$ and edges $T_S$, and let $K$ be the graph product where each vertex group is $H$. Clearly $G$ acts on 
 $K$ and one can see that $W\cong K \rtimes G$.

 \begin{prop}
 For every $n\ge 2$,the presentation 
  $$\gpr{s,t}{s^n,[s,s^{t^i}]\;i\ge 1}$$
  is a minimal presentation of $\f{Z}_n\wr \f{Z}$.
 \end{prop}

 \begin{proof}
  Clearly the relation $s^n$ is not redundant. For $i\ge 1$ let  $r_i=[s,s^{t^i}]$ and suppose that
  for some $m\ge 1$ $r_m$ is redundant.  Let 
  $$a_0=0\;\;,\;\;a_{2j}=j(m+1)+(1+\ldots+j)\;\;\;j\ge 1$$ and
  $$\begin{array}{cc}
       a_{2j+1}=\left\{
  \begin{array}{ccc}
  a_{2j} + j+1  &\text{if}& j< m-1 \\
  a_{2j} + j+2  &\text{if}& j\ge m-1
     
    \end{array}\right.
    
    &,\;\; j\ge 0
    \end{array}.
  $$
Note that 
$$
\begin{array}{cc}

a_{2j+1}-a_{2j}=\left\{
\begin{array}{ccc}
 j+1 & \text{if}& j<m-1 \\
 
j+2 & \text{if} & j\ge m-1
\end{array}\right.
&
,\;j\ge0
\end{array}
$$
and
$$|a_k-a_\ell|>m\;\;\text{if}\;\; |k-\ell| \ge 2. $$
  Finally let  $S=\{a_0,a_1,a_2,\ldots \}\subset  \f{Z}$ and observe that the set $S-S=\f{Z}\setminus \{-m,m\}$.
 Form the circle product $W=W(\f{Z}_n,\f{Z},S)$ with generators $x,y$. By the properties of $W$ we have
 for $i\ge 1$
 $$[x,x^{y^i}]=1\iff(0,i)\in T_S \iff i \in S-S \iff i\neq m.  $$
Therefore, under the assumption that $r_m$ is redundant in $\f{Z}_n\wr \f{Z}$, the map
$s\mapsto x,t\mapsto y$ defines a homomorphism from $\f{Z}_n\wr \f{Z}$ to $W$ which contradicts 
the fact that $r_m=1$ in $\f{Z}_n\wr \f{Z}$ but $[x,x^{y^m}]\neq 1$ in $W$.
 \end{proof}

\bibliographystyle{alpha}
\bibliography{Condensation.bib}
 
\end{document}